\documentclass[reqno]{amsart}
\usepackage{amssymb, amsmath}
\usepackage[a4paper,top=2 cm,bottom=2 cm,left=1.8cm, right=1.8cm]{geometry}
\usepackage{array}
\usepackage{xcolor}
\usepackage{tikz-cd}

\usepackage{comment}

\usepackage[hidelinks]{hyperref}
\usepackage{cite}
\usepackage{todonotes}

\makeatletter
\def\thmhead@plain#1#2#3{%
  \thmname{#1}\thmnumber{\@ifnotempty{#1}{ }\@upn{#2}}%
  \thmnote{ {\the\thm@notefont#3}}}
\let\thmhead\thmhead@plain
\makeatother

\newtheorem{theorem}{Theorem}
\newtheorem{corollary}[theorem]{Corollary}
\newtheorem{lemma}[theorem]{Lemma}
\newtheorem{proposition}[theorem]{Proposition}
\newtheorem{conjecture}[theorem]{Conjecture}

\theoremstyle{definition}

\newtheorem{definition}[theorem]{Definition}
\newtheorem{remark}[theorem]{Remark}

\numberwithin{equation}{section}
\numberwithin{theorem}{section}

\DeclareMathOperator{\id}{id}
\DeclareMathOperator{\End}{End}
\DeclareMathOperator{\spa}{span}

\newcommand{\F}{F\langle X\rangle}
\newcommand{\W}{W\langle X\rangle}
\DeclareMathOperator{\V}{var}
\DeclareMathOperator{\I}{Id}
\newcommand{\M}{\mathcal{M}}
\newcommand{\IM}{\mathcal{I}\mathcal{M}}
\DeclareMathOperator{\op}{op}

\DeclareMathOperator{\supp}{supp}

\begin{document}
\title{Multipliers, $W$-algebras and the growth of generalized polynomial identities}

\author{Fabrizio Martino}
\address{Fabrizio Martino, Dipartimento di Matematica e Informatica, Universit\`a degli Studi di Palermo, Via Archirafi 34, 90123, Palermo, Italy.}
\email{fabrizio.martino@unipa.it}

\author{Carla Rizzo}
\address{Carla Rizzo, CMUC, Department of Mathematics, University of Coimbra, 3000-143
Coimbra, Portugal.}
\email{carlarizzo@mat.uc.pt}

\thanks{F. Martino was partially supported by GNSAGA-INdAM}
\thanks{C. Rizzo was financially supported by the Fundação para a Ciência e a Tecnologia (Portuguese Foundation for Science and Technology) under the scope of the projects UID/00324/2025 (https://doi.org/10.54499/UID/00324/2025) (Centre for Mathematics of the University of Coimbra).}

\subjclass[2020]{Primary 16R10, 16R50; Secondary 16P90, 16G30}

\keywords{$W$-algebras, multipliers, matrix algebra, polynomial identity,  generalized polynomial identity, codimension growth, polynomial growth}

\begin{abstract}
Let $A$ be a $W$-algebra over a field $F$ of characteristic zero, where $W$ is any $F$-algebra. We first develop a comprehensive theory of generalized identities independent of the algebraic structure of $W$, using the multiplier algebra of $A.$ Then, we investigate the generalized variety generated by the $k\times k$ matrix algebra with a suitable action, proving that it exhibits almost polynomial growth of the generalized codimensions. Furthermore, we characterize the generalized varieties of almost polynomial growth generated by finite dimensional $W$-algebras. Finally, we provide a counterexample to the Specht property of generalized $T_W$-ideals in characteristic zero.
\end{abstract}

\maketitle

\section{Introduction}
Let $A$ and $W$ be associative algebras over a field $F$ of characteristic zero and let $\F$ be the free algebra generated by the countable set $X=\{x_1,x_2,\ldots\}.$ $A$ is called $W$-algebra if it is a $W$-bimodule such that
\begin{equation*}
		w(a_1 a_2)=(wa_1)a_2, \quad  (a_1 a_2)w=a_1(a_2 w), \quad (a_1 w) a_2= a_1 (w a_2).
	\end{equation*}
for all $a_1,a_2\in A$ and for all $w\in W.$ For instance, $A$ is an $A$-algebra if the action of $A$ on itself is given by left and right multiplication and, if $W\cong F,$ then such definition coincides with the usual one of $F$-algebra.

The free $W$-algebra is denoted by $\W$ and its elements are called $W$-polynomials or generalized polynomials. In particular, we may consider these elements as polynomials containing variables of $X$ and elements of $W$ as some sort of coefficients that may appear also between two or more variables. A $W$-polynomial is an identity for the $W$-algebra $A$ if it vanishes under all substitutions of its variables by elements of $A.$ The set $\I^W(A)$ of all generalized polynomial identities of $A$ is a $T_W$-ideal of $\W,$ i.e. a stable ideal by endomorphisms of $\W$ and, as in the usual case of $F$-algebras, one of the main problems is to find a set of generators of $\I^W(A).$ Nowadays, generators of $T_W$-ideals are known only in a few cases. The interested reader can find in \cite[Proposition 6.5.5]{BeidarMartindaleMikhalevm1996} the generalized identities of the full matrix algebra $M_k(F)$ for all $k\geq 1$ in case $W= M_k(F)$ acts by left and right multiplication, and in \cite{MartinoRizzo2024} the $T_W$-ideal of $UT_2(F),$ the $2\times 2$-upper triangular matrix algebra, in case $W=UT_2$ acts either as the subalgebra $D$ of diagonal matrices or the full algebra $UT_2$ by left and right multiplication.
These are the only known examples of generalized ideals at present.

Generalized identities were first introduced in Amitsur’s seminal paper \cite{Amitsur1965} in the context of primitive rings and later on they were studied by several authors from the perspective of the algebraic structure and properties of the rings satisfying them. We refer to \cite{BeidarMartindaleMikhalevm1996} and its bibliography for an exhaustive dissertation on such topic from this point of view. In more recent years a combinatorial approach arose, focusing on the study of generalized codimensions, which are essentially a quantitative tool for roughly measuring how many generalized identities are satisfied by an algebra. The codimension sequence $c_n(A)$ of a $F$-algebra $A$ was first introduced by Regev in \cite{Regev1972} as the dimension of the space of multilinear polynomials of degree $n$ reduced modulo the polynomial identities of $A.$ In the same paper, he also proved that such sequence is exponentially bounded if $A$ satisfies a non-trivial identity. Later, in 1999, Giambruno and Zaicev in \cite{GiambrunoZaicev1999} captured such exponential growth, providing an explicit method to compute it and solving the Amitsur conjecture, stating that the limit $\lim_{n\rightarrow +\infty} \sqrt[n]{c_n(A)}$ exists and is a non-negative integer called the PI-exponent of $A.$ 

Following this approach, Gordienko in \cite{Gordienko2010} defines the generalized codimension sequence and proves the Amitsur conjecture for any finite dimensional algebra $A$ nonetheless only in the case $W= A$ and the action is given by left and right multiplication. Furthermore, in \cite{MartinoRizzo2024} the authors study the generalized variety generated by $UT_2$ and prove that it has almost polynomial growth of the generalized codimensions if $W=UT_2$ acts either as a subalgebra isomorphic to $F$ or as the subalgebra $D$ by left and right multiplication.

So far, all results on generalized polynomial identities have been obtained by fixing in advance both the algebra $W$ and its action on $A$. In this paper, we take a different approach: we do not specify $W$ beforehand, and we allow its action to be more general than that arising from a subalgebra of $A$ acting by left and right multiplication. This shift is made possible by the weak action representability of the category of associative algebras (see \cite{Janelidze2022}). In particular, in our framework the action of $W$ emerges naturally as a unary operation, whereas in the classical setting it is treated as a nullary one. This perspective allows us to extend the notion of generalized identities to a significantly broader context and to develop a unified and more comprehensive theory.

The present paper is organized as follows. The first part aims to develop a complete, self-contained, and as general as possible theory of $W$-algebras that allows us to disregard the choice of $W.$ To reach this goal, we will introduce the notion of multiplier algebra $\M(A)$ and we establish a duality between $W$-actions and $\M(A)$-actions with an additional condition. It is worth mentioning that multipliers are an extremely versatile tool that appears in various and different areas of mathematics, such as noncommutative analysis in the context of $C^*$-algebras (see, for instance, \cite{DokuchaevExel2004, DokuchaevDelRioSimon2007, Exelbook}) or category theory (see \cite{BroxGarciaMancini2024, CigoliManciniMetere2023, MacLane1958}).

Then, after a brief dissertation about the consequences of the theory that has just been developed on algebras with polynomial identities, the second part is devoted to the study of the $W$-variety generated by $M_k(F)$, with the action of the entire $M_k(F)$ by left and right multiplication. In particular, after computing a basis for the T-ideal of identities of such an algebra — generalizing the result in \cite[Proposition 6.5.5]{BeidarMartindaleMikhalevm1996} — we prove an unexpected result, namely that such $W$-variety exhibits almost polynomial growth of the generalized codimensions.

Thereafter, a classification of the $W$-varieties of almost polynomial growth generated by finite dimensional algebras is provided. Such characterization was also made in other settings, including group graded algebras \cite{Valenti2011}, algebras with involution \cite{GiambrunoMishchenko2001}, special Jordan algebras \cite{Martino2019}, and algebras with derivation with some mild restrictions \cite{Argenti2024, Rizzo2023}.

Finally, the last part of the paper concerns some possible further directions. In particular, a counter-example to the Specht property in characteristic zero will be given in the last section, proving that the $T_W$-ideal of generalized identities of the unital Grassmann algebra $E$ is not finitely generated, if $W$ is not finitely generated and acts either as the finite dimensional Grassmann algebra $E_k$ or the full algebra $E$ by left and right multiplication.

\section{Multipliers and $W$-algebras}
Throughout this paper, $F$ will denote a field of characteristic $0$ and all algebras will be associative $F$-algebras.

Let $A$ be an algebra and $\End_F(A)$ be the algebra of the endomorphisms of $A$ acting on the left of $A.$ The multiplication in $\End_F(A)$ is given by the usual composition of endomorphisms, written by juxtaposition. Moreover, we denote by $\End_F(A)^{\op}$  the \emph{opposite algebra} of $\End_F(A)$,  which is the underlying vector space of $\End_F(A)$ endowed with the opposite product $\cdot^{\op}$ defined by $\psi_1\cdot^{\op} \psi_2:=\psi_2 \psi_1$ for all $\psi_1, \psi_2\in \End_F(A)$.

\begin{definition}\label{def: W-algebra}
    Given an $F$-algebra $W$, we say that $A$ is a \emph{$W$-algebra} if it is a $W$-bimodule, i.e., $A$ is both a left and a right $W$-module with the compatibility relation, for all $w_1,w_2\in W,$ $a\in A,$
\begin{equation}\label{eq: compatibility bimodule}
  (w_1a)w_2 = w_1(aw_2)
\end{equation}
  and also satisfies the following conditions, for all $w\in W,$ $a_1,a_2\in A$
	\begin{equation}\label{conditions W-algebra}
		 (a_1 a_2)w=a_1(a_2 w), \quad w(a_1 a_2)=(wa_1)a_2, \quad (a_1 w) a_2= a_1 (w a_2).
	\end{equation}
Moreover, if $W$ has a unity element $1_W$, we further assume that $1_W a=a 1_W=a$ for all $a\in A$. 
\end{definition}
For instance, if $W=F,$ then $A$ is just a $F$-algebra. Moreover, $W$ has itself the structure of $W$-algebra by taking the left and right $W$-actions as the usual left and right multiplication of $W.$


To formally describe the action of $W$ over an $F$-algebra $A,$ it is useful to introduce the notion of multiplier of $A$ (see \cite{MacLane1958}, where they are called bimultiplications).
Multipliers often appear in the context of $C^*$-algebras and in particular they are used to study partial group actions and the generalization of $C^*$-crossed products (see for instance \cite{DokuchaevExel2004, DokuchaevDelRioSimon2007}).

\begin{definition}
    Let $(R,L)\in\End_F(A)^{\op} \times \End_F(A).$  We say that $(R,L)$ is a  \emph{multiplier} of $A$ if for all $a_1,a_2\in A,$ one has that
    \begin{equation}\label{condition multiplier}
        R(a_1 a_2)= a_1R(a_2), \quad  L(a_1a_2)= L(a_1)a_2, \quad R(a_1)a_2=a_1L(a_2).
    \end{equation}
    We say that $R $ is a \emph{right multiplier} and $L$ is a \emph{left multiplier} of $A.$
\end{definition}

We denote by $\M(A)$ the set of all multipliers of $A$.
This set naturally carries the structure of a unital
$F$-algebra, where the sum and the product are defined component-wise and the unity is $(\id_A, \id_A).$

\begin{definition}
    The unital $F$-algebra $\M(A)$ is called the \emph{multiplier algebra} of $A.$
\end{definition}

Notice that for a fixed $m\in A,$ if one considers the endomorphism of $A$ of right and left multiplications by $m$:
\begin{equation*}\label{eq: L_m and R_m}
     R_m:a\mapsto am; \quad \quad L_m:a\mapsto ma,
\end{equation*}
for all $a\in A,$ then $(R_m,L_m)$ is a multiplier of $A$ called \emph{inner multiplier} of $A,$  and 
\begin{equation}\label{eq: omomorfismo A e M(A)}
    \begin{matrix}
    \mu: & A & \rightarrow & \M(A) \\
    \null & m & \mapsto & (R_m,L_m)
\end{matrix}
\end{equation}
is a homomorphism of $F$-algebras. Also, $\IM(A):=\mu(A)$ is a two-sided ideal of $\M(A),$ called \emph{inner multiplier ideal} of $A.$
Conversely, if $A$ has unity $1_A,$ then for any $(R,L)\in\M(A)$
$$
R(1_A)= R(1_A)1_A=1_AL(1_A)=L(1_A),
$$
thus letting $m=R(1_A)=L(1_A),$ we get
\begin{equation*}
    \begin{split}
        &R(a) = R(a1_A) = aR(1_A) = am\\
         &L(a) = L(1_A a)=L(1_A) a = ma
    \end{split}
\end{equation*}
for all $a\in A.$ Hence $(R,L)=(R_m,L_m)\in \IM(A).$  
This leads us to the following result.

\begin{proposition}\label{isomorfismo con algebra di mult}
$\mu$ is an isomorphism
    if and only if $A$ has unity.
\end{proposition}

To establish a correspondence between multipliers and $W$-actions,  it is necessary to ensure that the multipliers are permutable.
In other words, given two multipliers $(R_1,L_1)$ and $(R_2,L_2),$ we are wondering whether 
\begin{equation}\label{commutativita multipliers}
    R_{1}L_{2} = L_{2}R_{1} 
\end{equation}
hold or not.

In general this is not true; in fact, one can consider an algebra $A$ with trivial multiplication, i.e., $ab=0$ for all $a,b\in A.$ In this case, any couple of linear transformations $(f,g)$ is a multiplier of $A,$ but of course we cannot expect that \eqref{commutativita multipliers} may hold in such general case, since this should imply that in particular $fg=gf$ for all $f,g\in\End_F(A).$
Hence, we need an extra condition on $A.$

Clearly, any multiplier $(R,L)$ permutes with any inner multiplier.  Consequently, by Proposition \ref{isomorfismo con algebra di mult}, it follows that in the case of an algebra with unity, we have the following result. 

\begin{corollary}\label{mutiplier per algebre con unit}
    Let $A$ be a unital $F$-algebra. Then 
     $A\cong\IM(A)=\M(A)$
    and 
    $
    R_aL_b = L_bR_a
    $
    for all $a,b\in A.$
\end{corollary}

Unital algebras are not the only type of algebras for which the multiplier algebra satisfies the permutability property, as highlighted in the following proposition.

Recall that $A$ is said \emph{non-degenerate algebra} if there is no nonzero element $a\in A$ such that $ab = ba = 0,$ for all $b\in A,$ or equivalently if the homomorphism $\mu$ defined in \eqref{eq: omomorfismo A e M(A)}
is injective. Clearly, every unital algebra is non-degenerate.

\begin{proposition}[{\cite[Proposition 7.9]{Exelbook}}]
\label{pro: commutazione tra pezzi di mult}
    Let $A$ be either non-degenerate or idempotent, i.e., $A^2=A.$ Then 
    $
    R_2L_1 = L_1R_2,
    $
    for all $(R_1,L_1),(R_2,L_2)\in\M(A).$
\end{proposition}

We are now in a position to establish a connection between multiplier algebras and $W$-algebras.
To this end, recall that if $A$ is both a right and a left $W$-module, then the right and left actions of $W$ on $A$  induce the following $F$-algebras homomorphisms
\begin{equation*}
  \rho: W\rightarrow\End_F(A)^{\op}  \quad \text{ and }\quad \lambda:W\rightarrow\End_F(A)
\end{equation*}
where $\rho(w)(a) = aw$ and $\lambda(w)(a)= wa,$ for all $w\in W$ and $a\in A,$ called respectively \emph{right} and \emph{left representation} of $W$ on $A$, and vice versa. 

Now, let $A$ be a $W$-algebra. Since $A$  is both a left and right $W$-module, one can consider the corresponding right and left representations $\rho$ and $\lambda$ of $W$ on $A$. From \eqref{conditions W-algebra}, it directly follows that for all $w\in W$,
$$
\big( \rho(w), \lambda(w)\big)\in\M(A).
$$
Then, if one sets
\begin{equation}\label{omomorphism from W to M(A)}
    \begin{matrix}
    \Phi: & W & \rightarrow & \M(A) \\
    \null & w & \mapsto & \big(\rho(w), \lambda(w)\big)
\end{matrix}
\end{equation}
we get a homomorphism of $F$-algebras which, by \eqref{eq: compatibility bimodule}, satisfies the additional condition, for all  $w_1,w_2\in W,$
\begin{equation}\label{Phi permutability}
    \rho(w_2)\lambda(w_1)=\lambda(w_1)\rho(w_2).
\end{equation}

The converse is generally not true because not all multipliers of $A$ satisfy the condition \eqref{commutativita multipliers} of permutability. However, given an algebra $A$ and a homomorphism $\Phi:W\to \M(A)$ defined by $\Phi(w)=\big(R^{(w)},L^{(w)}\big)$ for all $w\in W$ such that the subalgebra $\overline{W}=\Phi(W)$ of $\M(A)$ satisfies the condition of permautability, i.e.,  $ R^{(w_1)}L^{(w_2)}=L^{(w_2)}R^{(w_1)}$ for all $w_1,w_2\in W,$ then $A$ inherits a natural structure of $W$-algebra by setting, for all $w\in W,$ $a\in A,$
\begin{equation*}
  aw := 
  R^{(w)}(a),   \quad wa := 
  L^{(w)}(a).
\end{equation*}
In particular, if $A$ is either non-degenerate or idempotent, then by Proposition \ref{pro: commutazione tra pezzi di mult} any homomorphism from $W$ to $\M(A)$ defines a structure of $W$-algebra on $A$.

Note that the above characterization aligns with the fact that the category of associative $F$-algebras is weakly action representable, as proved by Janelidze in \cite{Janelidze2022}. For an overview of the basic constructions, we also refer the interested reader to Section 2 of \cite{CigoliManciniMetere2023}.
Accordingly, we have the following definition. 

\begin{definition}[{\cite{Janelidze2022}}]
    The homomorphism $\Phi$ defined in \eqref{omomorphism from W to M(A)} with the property \eqref{Phi permutability} is called \emph{acting homomorphism} of $W$ on $A.$
\end{definition}

As a result, the action of $W$ on $A$ induces an action of the subalgebra $\overline{W}:=\Phi(W)$ of $\M(A)$ on $A.$ In other words, $A$  can be regarded as a $\overline{W}$-algebra. In this sense, we will say that $W$ acts on $A$ as the algebra $\overline{W}.$ 
 Furthermore, it is worth highlighting that \cite[Proposition 1.11]{BroxGarciaMancini2024}, closely related to \cite[Proposition 4.5]{Janelidze2022}, further establishes that for any $W$-algebra $A$ the acting homomorphism $\Phi$ such that $A$ is a $\overline{W}$-algebra is unique.

Since $A$ may admit different $W$-algebra structures corresponding to different acting homomorphisms $\Phi$, and hence to different images $\overline{W}\subseteq \M(A)$, whenever it is necessary to indicate the chosen action, we denote the $W$-algebra $A$ by $A^{\overline{W}}$.

\smallskip

Now, we state the following result on the structure of unital $W$-algebras. Its proof is an immediate consequence of the previous arguments, Corollary \ref{mutiplier per algebre con unit}, and the observation that, when $W$ is a unital algebra, $\Phi$ is a homomorphism of unital algebras, in particular satisfying $\Phi(1_W) = (\id_A, \id_A).$

\begin{proposition}\label{Wazioni per algebre con unit}
    Let $W$ be an $F$-algebra. If $A$ is a $W$-algebra with unity $1_A$, then the action of $W$ on $A$ is equivalent to the action of a suitable subalgebra $B$ of $A$ by left and right multiplication. Moreover, if $W$ is a unital algebra, then $1_A\in B$.
\end{proposition}

As a consequence, for a unital $W$-algebra $A$, we say that the $W$-action is given by left and right multiplication by a subalgebra $B\subseteq A$ if $\Phi(W)\cong B$. When it is necessary to specify the action, we denote $A$ by $A^B$.

\smallskip

Notice that the subalgebra $\overline{W}\subseteq \M(A)$ has a natural structure of $W$-algebra by left and right multiplication, more precisely, by defining $w\Bar{w}:=\Phi(w)\Bar{w}$ and $\Bar{w}w:=\Bar{w}\Phi(w)$, for all $w\in W$ and $\Bar{w}\in \overline{W}.$ 

Since $A$ is also a $\overline{W}$-algebra, we can define a multiplication on the direct sum of vector spaces $\overline{W}\oplus A$ as follows:
\begin{equation*}
(\Bar{w}_1,a_1) (\Bar{w}_2,a_2) :=
(\Bar{w}_1\Bar{w}_2, 
\Bar{w}_1 a_2 + a_1 \Bar{w}_2 + a_1 a_2)
\end{equation*}
where  $\Bar{w}_1 a_2=L_1(a_2)$ and $ a_1 \Bar{w}_2=R_2(a_1),$ with $  \Bar{w}_1=(R_1,L_1), \Bar{w}_2=(R_2,L_2)\in \overline{W}, a_1,a_2\in A.$ We refer to $\overline{W}\oplus A$ equipped with this multiplication as the \emph{semi-direct product} of $\overline{W}$ and $A,$ and we denote it by $\overline{W} \ltimes A$. 

Remark that this construction naturally leads to the following diagram:
\begin{equation*}\label{split extension}
    \begin{tikzcd}
0 \arrow[r] & A \arrow[r, "i_2"] & \overline{W} \ltimes A \arrow[r, "\pi_1", shift left] & \overline{W} \arrow[r] \arrow[l, "i_1", shift left] & 0
\end{tikzcd}
\end{equation*}
with $\pi_1(\Bar{w},a)=\Bar{w},$ $i_1(\Bar{w})=(\Bar{w},0)$ and $i_2(a)=(0,a),$ respectively. Since the above diagram satisfies that $\pi_1 \circ i_1=id_{\overline{W}}$ and $i_2(A)$ is in the kernel of $\pi_1,$ it is a split extension of $F$-algebras. 
For further details on split extensions of algebras and their relation with the (inner) action of $F$-algebras, we refer the interested reader to \cite{BorceuxJanelizedKelly2005, Janelidze2022}, where the equivalence between the two concepts is described in detail.

\begin{definition}
An subalgebra (ideal) $B$ of a $W$-algebra $A$ is called \emph{$W$-subalgebra} (\emph{$W$-ideal}) if it is $W$-invariant, i.e., $WB, BW\subseteq B$ where 
$ WB=\{ wb \ : \ w\in W, \ b\in B\}$ and $BW=\{ bw \ : \  b\in B, \ w\in W\}.$
\end{definition}

Thus, we have the following proposition that connects $W$-algebras (not necessarily unital) with unital $W$-algebras whose structure is well understood by the above proposition. 

\begin{proposition}\label{prop: semi-direct product}
    Let $A$ be a $W$-algebra, $\Phi$ be the acting homomorphism of $W$ on $A$ and denote $\overline{W}=\Phi(W)$. Then the semi-direct product $\overline{W} \ltimes A$ is a $W$-algebra such that we can identify $A$ with a suitable $W$-subalgebra  of $\overline{W}\ltimes A$. Moreover, if $\overline{W}$ is a unital, then $\overline{W} \ltimes A$ is also unital. 
\end{proposition}
\begin{proof}
    The algebra $\overline{W} \ltimes A$ naturally inherits the structure of $W$-algebra from the $W$-algebra structures of $\overline{W}$ and $A$. Specifically, for all $w \in W$, $\Bar{w} \in \overline{W}$, and $a \in A$, the right and left actions are defined as $(\Bar{w}+a)w:=\Bar{w}w+aw$ and  $w(\Bar{w}+a):=w\Bar{w}+wa$. Moreover, if we consider the homomorphism $i_2$ defined in the diagram above, then $i_2(A)$ is a $W$-subalgebra  of $\overline{W}\ltimes A$ isomorphic, as $W$-algebra, to $A$.
    Finally, if $\overline{W}$ is a unital, then $\overline{W} \ltimes A$ is a unital $F$-algebra, with the unity given by $(1_{\overline{W}},0).$ 
\end{proof}

\section{The Wedderburn-Malcev decomposition}

In this section, we describe the structure of finite-dimensional $W$-algebras. To this end, we first describe the action of multipliers on finite-dimensional algebras.

Let $A$ be a finite dimensional algebra and denote by $J=J(A)$ the Jacobson radical of $A$. Recall that $A$ is called \emph{semisimple} if and only if $J=0$, and $A$ is called \emph{simple} if it has no nontrivial ideals and $A^2= 0$.

Since the base field $F$ is of characteristic zero, by Wedderburn-Malcev Theorem for associative algebras (see \cite[Theorem 3.4.3]{GZbook}), there exists a unique, up to isomorphism, maximal semisimple subalgebra $B\subseteq A$ such that we can write $A$ as a direct sum of vector spaces
\begin{equation}\label{W-M decomposition}
    A=B+J.
\end{equation}
Then we have the following results.

\begin{proposition} \label{pro: multipliers and WM decomposition}
    Let $A$ be a finite dimensional algebra over a field $F$ of characteristic zero and $J$ be its Jacobson radical.
    Then $R(J), L(J)\subseteq J$ for all $(R,L)\in \M(A).$
\end{proposition}
\begin{proof}
   Let $A=B+J$ as in \eqref{W-M decomposition}, and take $(R,L)\in \M(A)$ and $j\in J$. Then there exist $b\in B$ and $j'\in J$ such that $R(j)=b+j'$. Hence, for all $\Bar{b}\in B$, $R(j)\Bar{b}=jL(\Bar{b})\in J$ since $J$ is an ideal of $A$ and by \eqref{condition multiplier}. As a consequence, $R(j)\Bar{b}=b\Bar{b}+j' \Bar{b}\in J$. Thus, given that $b\in B,$ $b \Bar{b}=0$ for all $\Bar{b}\in B.$ Since $B$ is semisimple, $b$ must be equal to $0$, and $R(j)=j'\in J$. Analogously, we can prove that $L(j)\in J$. 
\end{proof}

Similarly, we can prove the following.
\begin{proposition}\label{pro: simple algebras and multipliers}
  If $A$ is a finite dimensional algebra such that $A=B_1 \oplus \cdots \oplus B_k + J$, where $B_1, \ldots, B_k$ are simple algebras and $J$ is the Jacobson radical of $A$, then $R(B_i),L(B_i)\subseteq B_i+J$, for all $(R,L)\in \M(A)$ and $1\leq i \leq k.$
\end{proposition}

Notice that the Wedderburn-Malcev decomposition may not be $W$-invariant. In fact, in general $B$ in \eqref{W-M decomposition} is not $W$-subalgebra of $A.$ For example, if $J\neq 0$ and $W$ acts on $A$ as the subalgebra $J$ by left and right multiplication, then $W B + B  W \subseteq J$, that is, $B$ is not $W$-invariant. However, we have the following result.

\begin{theorem}\label{thm: WM decomposizion W-alg}
    Let $W$ be an algebra over a field $F$ of characteristic zero and $A$ be a finite dimensional $W$-algebra. Then the Jacobson radical $J$ of $A$ is a $W$-ideal of $A$. Moreover, if $F$ is algebraically closed, then $A=B_1\oplus \cdots\oplus B_k+J$, where $B_i$ is a simple algebra such that $WB_i, B_i  W \subseteq B_i+J$ for all $1\leq i \leq k.$
\end{theorem}
\begin{proof}
    Let $\Phi$ be the acting homomorphism of $W$ on $A$. Then $\Phi(w)=(\rho(w),\lambda(w))\in \M(A)$ for all $w\in W$. By Proposition \ref{pro: multipliers and WM decomposition}
 $\rho(w)(J), \lambda(w)(J)\subseteq J$ for all $w\in W$. Hence, by the definition of $\rho$ and $\lambda$, $jw, wj\in J$ for all $j\in J$ and $w\in W$. Thus, $J$ is a $W$-ideal of $A$.

 Now suppose that $F$ is algebraically closed of characteristic zero. Let $A=B+J$ as in \eqref{W-M decomposition}. Since $F$ is algebraically closed, by the Wedderburn-Artin Theorem (see \cite[Theorem 2.1.6]{Herstein1968}) we can write  $B=B_1 \oplus \cdots \oplus B_k,$ where $B_1, \ldots, B_k$ are simple algebras.
  Then
  $
  A=B_1 \oplus \cdots \oplus B_k+J.
  $
 As a consequence of Proposition \ref{pro: simple algebras and multipliers}, $\rho(w)(B_i), \lambda(w)(B_i)\subseteq B_i+J$ for all $w\in W.$ Thus, by the definition of $\rho$ and $\lambda,$ $bw, wb\in B_i+J$ for all $b\in B_i$ and $w\in W,$ and we are done.
\end{proof}

We say that a $W$-algebra $A$ is a \emph{$W$-simple} if $A^2\neq 0$ and $A$ has no non-zero $W$-ideals.

\begin{corollary}\label{cor: W-simple and simple algebras}
    Let 
    $A$ be a finite dimensional $W$-algebra over an algebraically closed field $F$ of characteristic zero.
    Then $A$ is $W$-simple if and only if it is simple.
\end{corollary}
\begin{proof}
    First, assume that $A$ is simple. Then clearly $A$ is also $W$-simple. So, suppose that $A$ is $W$-simple. By Theorem \ref{thm: WM decomposizion W-alg} $J$ is a $W$-ideal of $A$, then $J=0$. 
    Moreover, since $F$ is algebraically closed, by the second part of Theorem \ref{thm: WM decomposizion W-alg}
  we have that $A=B_1 \oplus \cdots \oplus B_k,$ where $B_1, \ldots, B_k$ are simple algebras and 
    $WB_i, B_i W \subseteq B_i$ for all $1\leq i \leq k$. Since $B_1, \ldots, B_k$ are also ideals of $A$, it follows that $k=1$ and $A$ is simple, as required.
\end{proof}

Let $M_{k}(F)$ be the algebra of $k \times k$ matrices over $F$, for some $k \geq 1.$ We write $M_{k}^B$ for the $W$-algebra $M_k(F)$ with the $W$-action given by left and right multiplication by a subalgebra $B\subseteq M_k(F)$.

\begin{corollary}\label{cor: W-simple algebras}
     If $A$ is a finite dimensional $W$-simple algebra over an algebraically closed field $F$ of characteristic zero, then $A$  is isomorphic, as $W$-algebra, to $M_{k}^B$ for some $k\geq 1$ and some subalgebra $B\subseteq M_k(F)$ containing $1_{M_k(F)}$.
\end{corollary}
\begin{proof}
    By Corollary \ref{cor: W-simple and simple algebras} it follows that $A$ is a simple algebra, and by Wedderburn Theorem (see \cite[Theorem 1.4.4]{Herstein1968}) it is isomorphic to $M_{k}(F)$ for some $k\geq 1$. Moreover, since $M_{k}(F)$ has unity, by Proposition \ref{Wazioni per algebre con unit} we are done.
\end{proof}

\section{Free $W$-algebra and generalized identities}\label{sezione4}


    Let from now on $W$ be a unital $F$-algebra.
    Recall that a homomorphisms $\varphi:A\rightarrow B$ between $W$-algebras $A,B$ must satisfy $\varphi(wav)=w\varphi(a)v$ for $a\in A$, $w,v\in W$.

    Since the class of $W$-algebras is a non-trivial variety in the sense of universal algebra, it contains the free $W$-algebra $\W$ freely generated by the countably infinite set of variables $X:=\{x_1,x_2, \ldots\}$, i.e., $\W$ is uniquely determined up to an isomorphism by the following universal property: given a $W$-algebra $A$, any map $\varphi: X\rightarrow A$ can be uniquely extended to a homomorphism of $W$-algebras $\Bar{\varphi}:\W\rightarrow A$, which we call the \emph{evaluation} of $\W$ at elements $\varphi(x_1),\varphi(x_2),\ldots$ from $A$. Notice that the evaluation $\Bar{\varphi}$ of $\W$ in $A$ depends not only on the map $\varphi$ but also on the right and left actions of $W$ on $A,$ and consequently on the corresponding right and left representations of $W$ on $A.$ Hence, $\Bar{\varphi}$ depends on the acting homomorphism $\Phi$ of $W$ on $A.$ Since $A$ can be a $W$-algebra with respect to different $W$-actions, we will denote the evaluation $\Bar{\varphi}$ by $\Bar{\varphi}_{\Phi}$, or simply $\varphi_{\Phi}$,  when necessary to clarify which action we are considering.

	We can describe $\W$ as follows: $\W$ is generated as an $F$-algebra by the set $\{v x_i w \ | \ i\geq 1, v,w\in W\}$ subject to the relations  
    \begin{equation*}
    \begin{split}
&1_W x_i 1_W = x_i, \ (v_1 + v_2)x_i w = v_1 x_i w + v_2 x_i w, \ v x_i (w_1+w_2)=v x_i w_1 + v x_i w_2, \ v(x_i+x_j)w = vx_i w + v x_j w
    \end{split}
    \end{equation*}
    for all $v,v_1,v_2,w,w_1,w_2\in W$ and $i,j\geq 1$, and the product given first by juxtaposition and then by multiplication in $W$.  Clearly, $\W$ has a natural structure of $W$-algebra by using the multiplication of $W$.
    
    Moreover, 
    given a basis $\mathcal{B}_W:=\{w_i\}_{i\in\mathcal I}$ of $W$,
    then a basis of $\W$ is the following
	$$
	\mathcal{B}_{\W}:=\left\lbrace w_{i_0}x_{j_1}w_{i_1}x_{j_2}\cdots w_{i_{n-1}}x_{j_n}w_{i_{n}} \mid j_1,\dots,j_n\geq 1 ,\, w_{i_0},\ldots,w_{i_{n}}\in\mathcal{B}_W, \, n\geq 1\right\rbrace.
	$$
Indeed, if $\mathcal{B}_W$ contains $1_W$, then it is clear that between any two variables always appears an element of the basis. Otherwise, $1_W$ can be written as a linear combination of some elements $w_{i_1},\ldots, w_{i_s}\in \mathcal{B}_W$, namely  $1_W= \alpha_1 w_{i_1}+\cdots + \alpha_s w_{i_s}.$ Thus, any monomial containing two or more adjacent variables can be expressed as a linear combination of monomials in which an element of $\mathcal B_W$ occurs between each pair of consecutive variables. For instance,
    $$
    w_jx_1x_2w_k = w_jx_1 1_W x_2w_k = w_jx_1(\alpha_1 w_{i_1}+\cdots + \alpha_s w_{i_s})x_2w_k = \alpha_1 w_jx_1w_{i_1}x_2w_k+\cdots + \alpha_s w_jx_1w_{i_s}x_2w_k.
    $$

    	The elements of $\W$ are called \emph{$W$-polynomials} or \emph{generalized polynomials} when there is no ambiguity about the role of $W$.
    A \emph{$T_W$-ideal} of the free $W$-algebra is a $W$-ideal which in addition is invariant under all $W$-algebra endomorphisms of $\W$, i.e., under the endomorphisms, called \emph{substitutions}, sending variables $x_i\in X$ to elements of $\W$.
     Moreover, given a set $S\subseteq\W$, a $W$- polynomial $f\in \W$ is said to be a \emph{consequence} of the $W$-polynomials in $S$ if $f$ belongs to the $T_W$-ideal generated by $S$. 

Note that the algebra $\W$, as defined here, is not the free product (i.e., the coproduct in the category of associative  $F$-algebras) of $W$ and the classical free $F$-algebra $F\langle X \rangle$, as considered in the classical theory of generalized polynomial identities (see, for instance, \cite[Chapter 6]{BeidarMartindaleMikhalevm1996}).
More precisely, the free product of $W$ and $F\langle X \rangle$ may be viewed as the free $W$-algebra with constant term, namely the semidirect product $W \ltimes \W$. A fundamental difference between $\W$ and the free product of $W$ and $F\langle X \rangle$ is that $W$ is not embedded in $\W$, while it is canonically embedded in the free product. 
This distinction stems from the greater generality of our framework:
our notion of $W$-algebras extends the concept of (not necessarily unital) $F$-algebras, whereas the classical theory of generalized identities deals with algebras that arise as a natural generalization of unital $F$-algebras.

\smallskip 

Given a $W$-algebra $A$, a generalized polynomial $f(x_1, \dots , x_n)\in \W$ is a \emph{$W$-identity} of $A$, or \emph{generalized identity} when the role of $W$ is clear, and we write $f\equiv 0$, if $f(a_{1},\dots,a_{n})=0$ for any $a_1,\ldots,a_n\in A$. 
We denote by $\I^W(A)$ the set of generalized identities of $A$, which is a $T_W$-ideal of the free $W$-algebra $\W$. Note that $\I^W(A)$ is the intersection of all kernels of evaluations of $\W$ from $A$.

\smallskip

For $n\geq 1$, we denote by $P_n^W$ the vector space of \emph{multilinear $W$-polynomials} in the variables $x_{1},\dots,x_{n}$, that is
    $$
	P_n^W:=\spa_F\{w_{i_0}x_{\sigma(1)}w_{i_1} x_{\sigma(2)}\cdots w_{i_{n-1}}x_{\sigma(n)}w_{i_{n}} \mid \sigma\in S_{n} ,w_{i_0},\ldots,w_{i_{n}}\in \mathcal{B}_W \},
	$$
	where $S_n$ denotes the symmetric group acting on $\{1,\ldots,n\}$.
	Since $F$ has characteristic zero, a standard argument shows that the $T_W$-ideal $\I^W(A)$ is completely determined by its multilinear generalized polynomials. 
	Thus it is reasonable to consider the quotient space
	$$
	P_n^W(A):= \dfrac{P_n^W}{P_n^W \cap \I^W(A)},
	$$
    and so one can define the $n$th \emph{$W$-codimension}, or \emph{generalized codimension}, of $A$ as
    $$
    c_n^W(A):=\dim_F P_n^W(A), \quad n\geq 1.
    $$
 Note that $c_n^W(A)$ is not necessarily finite, indeed we will show an example of infinite generalized codimensions in the last section. However, if at least one between $W$ and $A$ is finite dimensional, then $c_n^W(A)$ is finite for all $n \geq 1$, and it is interesting to ask if the limit 
$$
\exp^W(A):=\lim_{n\to \infty} \sqrt[n]{ c_n^W(A)}
$$
exists. 
 In case $W=F$, we are dealing with ordinary polynomial identities, and it is well known that the limit $\exp(A):=\exp^W(A)$ exists and is a nonnegative integer, called (ordinary) \emph{exponent} of $A$ (see \cite{GiambrunoZaicev1998, GiambrunoZaicev1999}).  In particular, if $A$ is a finite dimensional $F$-algebra, then
  \begin{equation*}
      \exp(A)=\max\{ \dim_F\big( B_{i_1} \oplus B_{i_2}\oplus \cdots \oplus B_{i_r} \big) \mid B_{i_1}J B_{i_2} J \cdots J B_{i_r}\neq 0,  \ 1\leq r \leq k, \ i_p \neq i_s,\ 1\leq p,s\leq r\},
  \end{equation*}
where $A = B_1 \oplus \cdots \oplus B_k + J$ with $B_1, \ldots, B_k$ simple algebras and $J = J(A)$ is the Jacobson radical of $A$.

  In \cite{Gordienko2010} Gordienko captured the exponential growth of the generalized codimension sequence of a finite dimensional unital algebra  $A$ that acts on itself by left and right multiplication. More precisely, he proved that if $A$ is a finite dimensional unital $A$-algebra, then $\exp^A(A)=\exp(A)$. 
Using the relationship between $W$-algebras and multiplier algebras, we can generalize this result as follows.

\begin{theorem}\label{thm: exponent}
   Let  $W$ be a unital algebra over a field $F$ of characteristic zero. If $A$ is a finite dimensional unital $W$-algebra,  then $\exp^W(A)$ exists and $\exp^W(A)=\exp(A)$.
\end{theorem}
\begin{proof}
    Since $A$ is a unital algebra, by Theorem \ref{Wazioni per algebre con unit} it follows that the action of $W$ on $A$ is equivalent to the action of a suitable unital subalgebra $B$ of $A$ by left and right multiplication. Then, since $A$ is finite dimensional, following step by step the proof of Theorem 3 in \cite{Gordienko2010} we get the desired conclusion.
\end{proof}

\begin{corollary}
    If $A$ is a finite dimensional unital $W$-algebra, then the sequence of generalized codimensions $c^W_n(A)$, $n\geq 1,$ grows exponentially or is polynomially bounded.
\end{corollary}

    A variety of $W$-algebras generated by a $W$-algebra $A$ is denoted by $\V^W(A)$ and is called \emph{$W$-variety}, or \emph{generalized variety}, and $\I^W(\mathcal{V}):=\I^W(A)$. The \emph{growth} of $\mathcal{V}= \V^W(A)$ is the growth of the sequence $c_{n}^W(\mathcal{V}):=c_{n}^W(A)$, $n\geq 1$. We say that the generalized variety $\mathcal{V}$ has \emph{polynomial growth} if $c_{n}^W(\mathcal{V})$ is polynomially bounded and $\mathcal{V}$ has \emph{almost polynomial growth} if  $c_{n}^W(\mathcal{V})$ is not polynomially bounded but every proper $W$-subvariety $\mathcal{U}$ of $\mathcal{V}$ has polynomial growth.

\smallskip

Let $A$ be a $W$-algebra such that $W$ acts on $A$ as a finite dimensional algebra. Then if we consider the acting homomorphism $\Phi: W \to \M(A)$ of $W$ on $A,$ 
we have that $\overline{W}= \Phi(W)$ is finite dimensional. Thus, since $F$ is of characteristic zero
by the Wedderburn-Malcev Theorem we can decompose 
$$\overline{W}=\overline{W}_{ss}+\Bar{J},$$ 
where $\overline{W}_{ss}$ is a semisimple subalgebra of $\overline{W}$ (isomorphic to $\overline{W}/\Bar{J}$) and $\Bar{J}=J(\overline{W})$ is the Jacobson radical of $\overline{W}$.

Now, denotes by $\pi:\overline{W} \rightarrow \overline{W}/\Bar{J}$ the canonical epimorphism. Since $\overline{W}_{ss}$ isomorphic to $\overline{W}/\Bar{J}$, we can assume that $\pi:\overline{W}\rightarrow \overline{W}_{ss}.$
Hence, if one defines the map $\Phi_{ss}: W \to \M(A)$ 
such that $\Phi_{ss}=\pi\Phi ,$ then $\Phi_{ss}$ is a homomorphism of $F$-algebra and determines another $W$-algebra structure on $A,$ which means that $\Phi_{ss}$ is another acting homomorphism of $W$ on $A.$ 
Thus, we denote by  $A^{\overline{W}_{ss}}$ the $W$-algebra $A$ under this new action, where $W$ acts on $A$ as the algebra $\overline{W}_{ss}=\Phi_{ss}(W).$ Using this notation, we have the following result.

\begin{theorem}\label{thm: semisimple action and algebra with 1}
    Let $W$ be a unital algebra over a field $F$ of characteristic zero and $A$ be a finite-dimensional unital 
    $W$-algebra.
     If $\Bar{J}\subseteq \IM(J(A))$ where $\Bar{J}=J(\overline{W})$ and $J(A)$ are the Jacobson radicals of $\overline{W}=\Phi(W)$ and $A$, respectively, then $A^{\overline{W}_{ss}}\in \V^{W}(A).$
\end{theorem}
\begin{proof}
 If $\Bar J=0$, there is nothing to prove, so we may assume that $\Bar J\neq 0$.

     Let $f\in \I^W(A)$  with $\deg f =n$ and assume, as we may, that $f$ is a multilinear. Decompose $f$ as $f=f_{{ss}}+f_{J}$, where $f_{{ss}}$ is a multilinear $W$-polynomial in which in each monomial appears only elements of $W$ that are sent through $\Phi$ into $\overline{W}_{ss},$ and $f_{J}$ is a multilinear $W$-polynomial in which in each monomial appears at least an element of $W$ that is sent through $\Phi$ into $\Bar{J}.$
     As a consequence, for all maps $\varphi:X \rightarrow A,$ we have that
       \begin{equation}\label{relation phi_W and phi_Wss}
        \varphi_{\Phi_{ss}}(f)=\varphi_{\Phi_{ss}}(f_{ss})= \varphi_{\Phi}(f_{ss}),
    \end{equation}
   and, since $f$ is a $W$-identity of $A$,
\begin{equation}\label{eq: W-evaluation fss and fJ}
       \varphi_{\Phi}(f_{ss})=-\varphi_\Phi(f_{J}).
    \end{equation}

Since $\Bar{J}\neq 0$, we may suppose that there exist a map $\varphi:X \rightarrow A$ such that $\varphi_\Phi(f_{J})\neq 0$. 
Since $A$ is a unital algebra, by Proposition \ref{Wazioni per algebre con unit} it follows that $\overline{W}=\overline{W}_{ss}+\Bar{J}$ is isomorphic to a subalgebra $B$ of $A$ that acts on $A$ by left and right multiplication. Since $\Bar{J}\subseteq \IM(J(A))$, the Jacobson radical $J(B)$ of $B$ is contained in $J(A)$. 
Thus we have that $\varphi_\Phi(f_{J})\in J(A)$, so by \eqref{eq: W-evaluation fss and fJ} also $\varphi_{\Phi}(f_{ss})\in J(A)$. Now notice that if $\varphi_{\Phi}(f_{ss})\in J(A)^q$ for some $q\geq 1$, then  by the definition of $f_{ss}$ and $f_J$ we have $\varphi_{\Phi}(f_{J})\in J(A)^{q+1}$. Thus, if $q$ is the biggest integer such that $\varphi_{\Phi}(f_{ss})\in J(A)^q$ and $\varphi_{\Phi}(f_{ss})\notin J(A)^{q+1}$, from \eqref{eq: W-evaluation fss and fJ} it follows that $ \varphi_{\Phi}(f_{ss})=-\varphi_{\Phi}(f_{J})= 0,$ a contradiction. Therefore we have that $\varphi_\Phi(f_{J})=0$ for all map $\varphi:X \rightarrow A.$ By \eqref{eq: W-evaluation fss and fJ} and \eqref{relation phi_W and phi_Wss}  we deduce that $ \varphi_{\Phi_{ss}}(f)=0$ for all $\varphi:X \rightarrow A.$ Thus, $f\in \I^W (A^{\overline{W}_{ss}})$, as required.
\end{proof}

\section{$W$-identities of $M_k(F)$}

The problem of finding generators for the \( T \)-ideal of identities of the matrix algebra \( M_k(F) \), for every \( k \geq 3 \), is one of the main open problems—if not the most important one—in the theory of algebras with polynomial identities. It is worth recalling that, for \( k = 2 \), a basis for the \( T \)-ideal was found in \cite{Drensky1981, Koshlukov2001, Razmyslov1973}, while nothing is known for \( k \geq 3 \). From this point of view, until now the only significant exceptions have been given by the \( T \)-ideal of trace identities of \( M_k(F) \), whose generator, the Cayley-Hamilton polynomial, was determined in \cite{Procesi1976, Razmyslov1974}, by the ideal of the differential identities of $M_k(F)$ with the action of its full Lie algebra of derivations (see \cite{BroxRizzo2024}),
and by the $T_W$-ideal of the $W$-algebra $M_k(F)$ when \( W \) acts as the full algebra \( M_k(F) \) by left and right multiplication, which was computed in \cite[Proposition 6.5.5]{BeidarMartindaleMikhalevm1996}.

This section has a twofold purpose. On the one hand, we re-establish Proposition 6.5.5 of \cite{BeidarMartindaleMikhalevm1996}, we translate it into our language of $W$-algebras and multipliers, and, in particular, we determine a basis for $P_n^W(M_k(F))$ and the corresponding $W$-codimension sequence. On the other hand, we prove a result that is, in some sense, even more surprising, namely that the variety of \( W \)-algebras generated by \( M_k(F) \), endowed with the above action, has almost polynomial growth.

As a first step, we properly define the notation that will be used throughout this section. Since we are assuming that $W$ acts as the full $k\times k$ matrix algebra, that is, $\Phi(W)\cong M_k(F)$, where $\Phi$ is the acting homomorphism of $W$ on $M_k(F)$, we consider a basis
$$
\mathcal B_W =\{w_{ij} \ | \ i,j\geq 1\}
$$
of $W$ such that $\Phi(w_{ij}) = (R_{e_{ij}}, L_{e_{ij}})$ for all $1\leq i,j\leq k$ and $w_{pq}\in \ker\Phi$ if at least one $p$ or $q$ greater than $k$. We denote by $M_k$ the algebra $M_k(F)$ endowed with this action.

    Let $[x_1,x_2]:=x_1x_2-x_2x_1$ be the {\em commutator} of $x_1$ and $x_2$.   A straightforward computation shows that the following $W$-polynomials are $W$-identities of $M_k$: 
\begin{equation}\label{preconseguenze}
    [w_{11}x_1w_{11},w_{11}x_2w_{11}], \quad w_{ij}w_{ml}x-\delta_{jm}w_{il}x, \quad xw_{ij}w_{ml}-\delta_{jm}xw_{il},
\end{equation}
    where $i,j,l,m\in \{1,\ldots, k\}$ and $\delta_{jm}$ denotes the Kronecker delta. We now establish some consequences needed to reach our goal. 

\begin{lemma}\label{conseguenzastrana}
    For all $i,j,l,m,p,q\in\{1,\ldots, k\}$ the $W$-polynomials
    $$
    w_{ij}x_1w_{lm}x_2w_{pq} - w_{im}x_2w_{pj}x_1w_{lq}
    $$
    are consequences of the polynomials in \eqref{preconseguenze}.
\end{lemma}
\begin{proof}
    In $[w_{11}x_1w_{11},w_{11}x_2w_{11}] = w_{11}x_1w_{11}x_2w_{11}-w_{11}x_2w_{11}x_1w_{11}$ we substitute $x_1$ by $w_{1j}x_1w_{l1}$ and $x_2$ by $w_{1m}x_2w_{p1},$ then by multiplying on the left by $w_{i1}$ and in the right by $w_{1q},$ we get the desired conclusion.
\end{proof}

\begin{theorem}\label{identitadellematriciecodimensioni}
   Let $M_k$, with $k\geq 2$,  be the $W$-algebra $M_k(F)$ of $k\times k$ matrices where $W$ acts on it as  $M_k(F)$ itself by left and right multiplication. 
    Then $\I^W(M_k)$ is generated as $T_W$-ideal by the following $W$-polynomials:
    $$
    [w_{11}x_1w_{11},w_{11}x_2w_{11}], \quad w_{ij}w_{ml}x-\delta_{jm}w_{il}x, \quad xw_{ij}w_{ml}-\delta_{jm}xw_{il}, \quad w_{pq}x, \quad xw_{pq},
    $$
    where $i,j,l,m\in \{1,\ldots, k\}$, $\delta_{jm}$ denotes the Kronecker delta and $p,q\geq 1$ with at least one $p$ or $q$ greater than $k$. Moreover, $c_n^W(M_k) = k^{2(n+1)}$ for all $n\geq 1.$
\end{theorem}
\begin{proof}
    Let $Q$ be the $T_W$-ideal generated by the above polynomials.It is clear that $Q\subseteq \I^W(M_k).$

    In order to prove the opposite inclusion, first notice that by Lemma \ref{conseguenzastrana}, any multilinear monomial of $P_n^W$ can be written modulo $Q$ as a linear combination of monomials of the type
    \begin{equation}\label{relativamenteliberamatrici}
        w_{i_0j_0}x_1w_{i_1j_1}x_2\cdots w_{i_{n-1}j_{n-1}}x_nw_{i_nj_n}
    \end{equation}
    where  $1\leq i_r,j_r\leq k$ for all $0\leq r \leq n$. We claim that the above monomials are linearly independent modulo $P_n^W\cap \I^W(M_k).$ To this end, let $f\in P_n^W\cap \I^W(M_k)$ be a linear combination of the above monomials and write
    $$
    f\equiv  \sum_{I, J, i_0, j_n}\alpha_{I, J, i_0, j_n}w_{i_0j_0}x_1w_{i_1j_1}x_2\cdots w_{i_{n-1}j_{n-1}}x_nw_{i_nj_n} \quad \text{(mod }\I^W(M_k)),
    $$
    where $I=\{i_1,\ldots, i_n\}$ and $J=\{j_0,\ldots, j_{n-1}\}$, with $1\leq i_r,j_r\leq k$ for $0\leq r \leq n$. Let suppose by contradiction that there exist $I$ and $J$ such that $\alpha_{I, J, i_0, j_n}\neq 0$ for some $i_0$ and $j_n.$ Then by making the evaluation $\varphi(x_s)= e_{j_{s-1}i_s}$ for all $1\leq s\leq n$ we get
    $$
    \varphi(f) = \sum_{1\leq i_0,j_n\leq k}\alpha_{I, J, i_0, j_n} e_{i_0j_n}=0,
    $$
    therefore $\alpha_{I, J, i_0, j_n}=0$ for all $i_0$ and $j_n$, a contradiction. Thus the monomials in \eqref{relativamenteliberamatrici} are linearly independent modulo $P_n^W\cap\I^W(M_k)$ and since $P_n^W\cap\I^W(M_k)\supseteq P_n^W\cap Q,$ they form a basis of $P_n^W$ modulo $P_n^W\cap\I^W(M_k)$ and so $Q= \I^W(M_k).$

    Finally, by counting the monomials in \eqref{relativamenteliberamatrici} we get
    $$
    c_n^W(M_k) = (k^2)^{n+1},
    $$
    as claimed.
\end{proof}

It is worth mentioning that the same result about the $W$-codimension sequence was achieved by Gordienko in \cite[Theorem 1]{Gordienko2010} with combinatorial methods.

Next, we shall prove that $\V^W(M_k)$ has almost polynomial growth. To this end, we need the following technical lemma.

\begin{lemma}\label{identitaaggiuntiva}
    Let $\mathcal{V}\subseteq \V^W(M_k).$ Then $\mathcal{V}$ is a proper subvariety if and only if there exists an integer $t\geq 1$ such that
    $$
    w_{11}x_1w_{11}x_2\cdots w_{11}x_tw_{11}\in \I^W(\mathcal{V}).
    $$
\end{lemma}
\begin{proof}
    Let us assume that $w_{11}x_1w_{11}x_2\cdots w_{11}x_tw_{11}\in \I^W(\mathcal{V}).$ Since $w_{11}x_1w_{11}x_2\cdots w_{11}x_tw_{11}\notin\I^W(M_k)$ for all $t\geq 1,$ then it is clear that $\mathcal{V}\subsetneq \V^W(M_k).$

    Conversely, suppose that $\mathcal{V}$ is a proper subvariety of $\V^W(M_k).$ Then there exists a multilinear polynomial $f\in \I^W(\mathcal{V})$ of degree $t$ that is not a $W$-identity of $M_k.$ Arguing as in the proof of Lemma \ref{relativamenteliberamatrici}, we can write
    $$
    f\equiv \sum_{I, J, i_0, j_t}\alpha_{I, J, i_0, j_t}w_{i_0j_0}x_1w_{i_1j_1}x_2\cdots w_{i_{t-1}j_{t-1}}x_tw_{i_tj_t} \quad \text{(mod }\I^W(M_k)),
    $$
  for some scalar $\alpha_{I, J, i_0, j_t}\in F$, not all zero, where $I=\{i_1,\ldots, i_t\}$ and $J=\{j_0,\ldots, j_{t-1}\}$, with $1\leq i_r,j_r\leq k$ for $0\leq r \leq n$.
    
    Suppose that $\alpha_{I, J, i_0, j_t}\neq 0$ for some $I,J,i_0,j_t$. 
    By substituting $x_l$ with $w_{j_{l-i}1}x_lw_{1,i_l},$ for all $1\leq l\leq t$ and then by multiplying on the left by $w_{1i_0}$ and on the right by $w_{j_t1},$ we get the desired conclusion.
\end{proof}

We are now in a position to prove the second main result of this section.

\begin{theorem}\label{matriciAPG}
     Let $M_k$, with $k\geq 2$,  be the $W$-algebra $M_k(F)$ of $k\times k$ matrices where $W$ acts on it as  $M_k(F)$ itself by left and right multiplication.  Then $\V^W(M_k)$ has almost polynomial growth.
\end{theorem}
\begin{proof}
    Let $\mathcal{V}$ be a proper subvariety of $\V^W(M_k).$ Then, by Lemma \ref{identitaaggiuntiva}, there exists an integer $t\geq 1$ such that $w_{11}x_1w_{11}x_2\cdots w_{11}x_tw_{11}\in \I^W(\mathcal{V}).$ If we substitute in the latter one $x_l$ with $w_{1j_{l-1}}x_lw_{i_l1},$ for all $1\leq l\leq t$ and then we multiply on the left by $w_{i_01}$ and on the right by $w_{1j_t}$ we get the consequence
    $$
    w_{i_0j_0}x_1w_{i_1j_1}\cdots w_{i_{t-1}j_{t-1}}x_tw_{t_ij_t}\in \I^W(\mathcal{V}),
    $$
    for all $1\leq i_r,j_r\leq k$, $0\leq r \leq n$. It immediately follows that $c_n^W(\mathcal{V}) = 0$ for all $n\geq t$ and in general $c_n^W(\mathcal{V})\leq k^{2t+2}$ for all $n\geq 1.$ Hence $\V^W(M_k)$ has almost polynomial growth and we are done.
\end{proof}

Note that, in view of the $W$-identities they satisfy, it follows that $M_h\notin\V^W(M_k)$ and $M_k\notin\V^W(M_h)$ for all $k\neq h.$ Therefore, the $W$-algebras of type $M_k$ generate pairwise distinct varieties of almost polynomial growth for every $k\geq 2.$

\section{$W$-identities of $2\times 2$ upper triangular matrix algebra}
In this section we present the results in \cite{MartinoRizzo2024} about generalized identities of the algebra $UT_2$ of $2 \times 2$ upper triangular matrices. Notice that although such results were obtained in the case of $UT_2$-algebras, the present work seeks to generalize and extend the findings to the case where $W$ is an arbitrary unital algebra. 

We start with the following simple remark.

\begin{remark}\label{sottoalgebre di dimensione 2}
If $B$ is a unital subalgebra of $UT_2$ of dimension $2,$ then either $B\cong D$ or $B\cong C,$ where $D=Fe_{11}\oplus Fe_{22}$ and $C=F 1_{UT_2}+ Fe_{12}.$ Indeed, if $B$ is semisimple, then it is clear that it must be $B\cong F\oplus F,$ that is $B\cong D.$ If $B$ is not semisimple, then a basis can be chosen so that it contains an element of the radical, say $j=\alpha e_{11}+\beta e_{22}+\gamma e_{12},$ for some $\alpha,\beta,\gamma\in F.$ If $\alpha\neq 0$ or $\beta\neq 0$ then $j$ can not lie in the radical since it would be non-nilpotent. Hence $j=\gamma e_{12}$ and $B\cong C.$
\end{remark}

As a direct consequence of the previous Remark plus Proposition \ref{Wazioni per algebre con unit}, we get that on $UT_2$ one can define four non-equivalent structures of $W$-algebra according to which unital subalgebra of $UT_2$ is acting by left and right multiplication. We will denote them by $UT_2^F,$ $UT_2^D,$ $UT_2^C$ and simply $UT_2$ depending on whether $\Phi(W)$ is isomorphic to $F,$ $D,$ $C$ or the full algebra $UT_2,$ respectively, where $\Phi$ is the acting homomorphism of $W$ on $UT_2$.

  In $UT_2^F$, since $\Phi(W)\cong F$, we fix a basis $\mathcal{B}_W:=\{w_i\}_{i\in I}$  of $W$ such that $w_0=1_W$ and $w_i\in \ker \Phi$ for all $i \geq 1$. Since we are dealing with $UT_2$ with the ordinary structure of $F$-algebra, the ideal of identities was computed in \cite{Malcev1971} and the following theorem holds.

\begin{theorem}\label{identita e codimensioni di UT2F}
    Let $UT_2^F$ be the $W$-algebra 
    $UT_2$ where $W$ acts on it as the algebra $F$ by left and right multiplication.
    Then $\I^W(UT_2^F)$ is generated as $T_W$-ideal by the $W$-polynomials:
    $$
    w_i x, \quad xw_i, \quad [x_1,x_2][x_3,x_4],
    $$
    for all $i\geq 1.$ Moreover, $c_n^W(UT_2^F) = (n-2)2^{n-1}+2.$
\end{theorem}

In $UT_2^D$, since $\Phi(W)\cong D$, we choose a basis $\mathcal{B}_W:=\{w_i\}_{i\in I}$ of $W$ such that $w_0=1_W$, $\Phi(w_1)=(R_{e_{22}},L_{e_{22}})$ and $w_i\in \ker \Phi$ for all $i \geq 2$. Since we are dealing with $UT_2$ equipped with a $D$-algebra structure, \cite[Theorem 5.1]{MartinoRizzo2024} yields the following result.

\begin{theorem}\label{identita e codimensioni di UT2D}
    Let $UT_2^D$ be the $W$-algebra 
    $UT_2$ where $W$ acts on it as the algebra $D$ by left and right multiplication.
   Then $\I^W(UT_2^D)$ is generated as $T_W$-ideal by the $W$-polynomials:
    $$
    w_i x, \quad xw_i, \quad  w_1^2 x- w_1 x, \quad  x w_1^2 - x w_1, \quad [x_1,x_2]-[x_1,x_2,w_1],
    $$
    for all $i\geq 2.$
    Moreover, $c_n^W(UT_2^D) = n2^{n-1}+2.$
\end{theorem}

 In $UT_2$, since $\Phi(W)\cong UT_2$, consider a basis $\mathcal{B}_W:=\{w_i\}_{i\in I}$ of $W$ such that $w_0=1_W$, $\Phi(w_1)= (R_{e_{22}},L_{e_{22}}),$ $ \Phi(w_2)=(R_{e_{12}},L_{e_{12}})$ and $w_i\in \ker \Phi$ for all $i \geq 3$. Since $UT_2$ is endowed with a $UT_2$-algebra structure in this case, from \cite[Theorem 3.2]{MartinoRizzo2024} we obtain the following result.

\begin{theorem}\label{identita e codimensioni di UT2full}
    Let $UT_2$ be the $W$-algebra  
    $UT_2$ where $W$ acts on it as $UT_2$ itself  by left and right multiplication. 
    Then $\I^W(UT_2)$ is generated as $T_W$-ideal by the $W$-polynomials:
     \begin{align*}
     &  w_i x, \quad xw_i, \quad w_1^2x - w_1 x, \quad  x w_1^2 - x w_1, \quad w_2^2 x, \quad xw_2^2,\quad w_1w_2x, \\
     & xw_1w_2, \quad  w_2w_1x-w_2x,  \quad xw_2w_1-xw_2,  \quad [x_1,x_2]-[x_1,x_2,w_1], 
    \end{align*}
    for all $i\geq 3.$ Moreover, $c_n^W(UT_2) = (n+2)2^{n-1}+2.$
\end{theorem}
 Furthermore, in \cite{MartinoRizzo2024} it was also proved that while $\V^W(UT_2^F)$ and $\V^W(UT_2^D)$ generate two distinct varieties of $W$-algebras of almost polynomial growth, the same does not hold true for $\V^W(UT_2)$ since it contains $UT_2^D.$ We can also remark that by Theorem \ref{thm: semisimple action and algebra with 1}, $UT_2^F\in\V^W(UT_2^C)$ thus $\V^W(UT_2^C)$ is not also of almost polynomial growth.

 Just for the sake of completeness, here we state the result concerning $T_W$-ideal and $W$-codimensions of $UT_2^C.$ This can be proven following step-by-step the proof in \cite[Theorem 3.2]{MartinoRizzo2024} with the necessary yet intuitive modifications. 
 Here, we take a basis $\mathcal{B}_W:=\{w_i\}_{i\in I}$ of $W$ such that $w_0=1_W$, $\Phi(w_1)=(R_{e_{12}},L_{e_{12}})$ and $w_i\in \ker \Phi$ for all $i \geq 2$.

 \begin{theorem}
     Let $UT_2^C$ be the $W$-algebra 
     $UT_2$ where $W$ acts on it as the algebra $C$ by left and right multiplication.
     Then $\I^W(UT_2^C)$ is generated as $T_W$-ideal by the $W$-polynomials:
    $$
    w_i x, \quad xw_i, \quad w_1^2 x, \quad x w_1^2,  \quad w_1[x_1,x_2], \quad [x_1,x_2]w_1, \quad [x_1,x_2][x_3,x_4],
    $$
    for all $i\geq 2.$ Moreover, $c_n^W(UT_2^C) = n2^{n-1}+2.$
 \end{theorem}

\section{Classifying almost polynomial growth $W$-varieties}

 In this section, we present a characterization of the generalized varieties of almost polynomial growth, starting with the following preliminary results.

\begin{proposition}\label{pro: UT_n e UT_2}
Let $F$ be a field of characteristic zero and let us consider the $W$-algebra $UT_k$ of upper triangular matrices of order $k$ over $F$ for some $k >1.$ Then $\V^W(UT_k)$ contains at least one of the algebras in $\left\{UT_2^F, UT_2^D\right\}.$
\end{proposition}
\begin{proof}
Since $UT_k$ is a unital algebra, Proposition \ref{Wazioni per algebre con unit} ensures that the action of $W$ is equivalent to the action of a subalgebra $B\subseteq UT_k$ by left and right multiplication and $1_{UT_k}:=e_{11}+\cdots + e_{kk}\in B$. Moreover, by Theorem \ref{thm: semisimple action and algebra with 1}, we may assume without loss of generality that $B$ is semisimple. Then there exist positive integers $(t_1,\ldots,t_m)$ such that $\sum_{i=1}^m t_i=k$ and 
$$B\cong F(e_{i_1 i_1} + \cdots + e_{i_{t_1} i_{t_1}})\oplus \cdots \oplus F(e_{i_{t_1+\cdots +t_{m-1}+1} i_{t_1+\cdots +t_{m-1}+1}} + \cdots + e_{i_{t_1+\cdots +t_m} i_{t_1+\cdots +t_m}})$$ where $\sum_{j=1}^k e_{i_{j} i_{j}}=1_{UT_k}$. 

    If $m=1,$ then $B\cong F$ and we are dealing with the ordinary case. Consequently, we have  $UT_2^F\in\V^W(UT_k).$  
    Let us assume then that $m\geq 2.$
    Clearly, the vector space $ V=\spa_F\{ e_{i_1 i_1}, \, e_{i_{t_1+1} i_{t_1+1}} , \, e_{i_1 i_{t_1+1}}\}$ forms a $W$-subalgebra of $UT_k$ isomorphic, as W-algebra, to $UT_2^D. $
    Therefore, we conclude that $UT_2^D \in \V^W(UT_k),$ completing this case.
\end{proof}

Recall that, for an ordered set of positive integers $(t_1,\ldots,t_m)$, the upper block triangular matrix subalgebra of the matrix algebra $M_{t_1+\cdots+t_m}(F)$ is the algebra
$$
UT(t_1,\ldots,t_m)=\begin{pmatrix}
    U_{11} & U_{12} & \cdots & U_{1m}\\
   & U_{22} & \cdots & U_{2m}\\
    &  & \ddots   & \vdots\\
     &  &  & U_{mm}
\end{pmatrix}
$$
where $U_{ii}=M_{t_i }(F)$ for $1\leq i \leq m$, $U_{ij}$ is a rectangular $t_i \times t_j$ matrix over $F$ for $1\leq i< j \leq m$, and all the others entries are $0$. We refer to $U_{ii}$, $1\leq i \leq m$, as diagonal blocks. 
In particular, for $m=1$, we have just one diagonal block and $UT(t_1)=M_{t_1}(F)$. Clearly, $UT(t_1,\ldots,t_m)\cong M_{t_1}(F)\oplus \cdots \oplus M_{t_m}(F) + J$ where $J\cong\oplus_{1\leq i<j\leq k}U_{ij}$ is the Jacobson radical. Then, as a consequence of the theorem of block triangularization \cite[Theorem 1.5.1]{RadjaviRosenthalBook}, we have the following.

\begin{proposition}\label{pro: block triangularization}
Let $B$ be a subalgebra of the algebra $M_{k}(F)$ of matrices of order $k>1$ over $F$ containing $1_{M_k(F)}$. Then there exist positive integers $(t_1,\ldots,t_m)$ such that $\sum_{i=1}^m t_i=k$ and  $B$ is isomorphic to a subalgebra of $UT(t_1,\ldots,t_m)$ with the diagonal blocks being full matrix algebras of the corresponding dimensions.
\end{proposition}

\begin{lemma}\label{lem: AiJAk in J^q}
Let $A=A_1\oplus \cdots \oplus A_m + J$ be a finite dimensional unital algebra, where $A_1 \cong \cdots \cong  A_m \cong F,$ and suppose that $A$ is a $W$-algebra. 
If there exist $1\leq r,s\leq  m$, $r \neq s$, such that $A_r JA_s\neq 0$, then  $\V^W(A)$ contains at least one of the algebras in $\left\{UT_2^F, UT_2^D\right\}.$ 
\end{lemma}
\begin{proof}
    Since $A$ is unital, by Theorem \ref{isomorfismo con algebra di mult} the action of $W$ on $A$ is equivalent to the action of a suitable unital subalgebra $B$ of $A$ by left and right multiplication. Additionally, by Theorem \ref{thm: semisimple action and algebra with 1}, we can assume that $B$ is semisimple, i.e., $B\cong A_{i_1}\oplus \cdots \oplus A_{i_p}$ with $i_1, \ldots, i_p$ distinct elements in $\{1,\ldots, m\}.$ 

    Now, since $A_rJA_s\neq 0,$ there exist an element $j\in J$ such that $e_rje_s\neq 0,$ where $e_r$ and $e_s$ denote the unit elements of $A_r$ and $A_s,$ respectively. 
    For all $b\in B,$ let
    $$
     be_r=\alpha_b^{(r)} e_r, \quad e_rb=\beta_b^{(r)} e_r, \quad be_s=\alpha_b^{(s)} e_s, \quad 
     e_s b=\beta_b^{(s)} e_s
    $$
    for some $\alpha_b^{(r)},\beta_b^{(r)}, \alpha_b^{(s)} , \beta_b^{(s)}\in F.$ As a consequence, for all $b\in B,$ we obtain
    $$
    be_rj e_s=(be_r)e_rj e_s=\alpha_b^{(r)} e_rj e_s \quad \text{ and } \quad  e_rj e_s b=e_rj e_s(e_s b)=\beta_b^{(s)} e_rj e_s.
    $$
    Now, consider the vector space
    $A'=\spa_F\{e_r, \ e_s, \ e_r j e_s \}.$ 
    Clearly, $A'$ is a unital $W$-subalgebra of $A.$ Again by Theorems \ref{isomorfismo con algebra di mult} and \ref{thm: semisimple action and algebra with 1}, the action of $W$ on $A'$ is equivalent to the action of a semisimple unital subalgebra $B'$ of $A'$ by left and right multiplication. Consequently, we have $\dim_F B'=1 $ or $2.$
If $\dim_F B'=1 $, then $B'\cong F$ and $A'$ is isomorphic to $ UT_2^F$ as $W$-algebras. So, suppose instead that $\dim_F B'=2 .$ Then $1_{B'}=1_{A'}$, and since $B'$ is semisimple, we can take $\{1_{B'}, e_s\}$ as a basis of $B'.$ As a result, $A'$ is isomorphic to $UT_2^D$ as $W$-algebras. Thus, we conclude that either $UT_2^F\in \V^W(A')$ or $UT_2^D\in \V^W(A').$ Since $\V^W(A')\subseteq\V^W(A),$ the lemma is proved.
\end{proof}

\begin{lemma}\label{lem:unital W-algebra PG}
    Let $W$ be a unital algebra over a field $F$ of characteristic zero and let
    $A$ be a finite dimensional unital $W$-algebra. 
    If $UT_2^F, UT_2^D, M_k  \notin \V^W(A)$ for all $k\geq 2$, then the generalized codimension sequence $c_n^W(A)$ is polynomially bounded.
\end{lemma}
\begin{proof}
    Using an argument analogous to that used in the ordinary case (see \cite[Theorem 4.1.9]{GZbook}), we can prove that generalized codimensions do not change upon extension of the base field. Therefore, we may assume $F$ is algebraically closed. 
    
    Since the Jacobson radical $J=J(A)$ of $A$ is a $W$-ideal,  $\Bar{A}= A/J$ is a $W$-algebra. Moreover, since $\Bar{A}$ is semisimple and by Corollary \ref{cor: W-simple and simple algebras}, we have that
    $\Bar{A}=\Bar{A}_1 \oplus \cdots \oplus \Bar{A}_m$
    where $\Bar{A}_1, \ldots , \Bar{A}_m$ are $W$-simple algebras.  By Corollary \ref{cor: W-simple algebras}, for all $1\leq i \leq m$, $\Bar{A}_i$ is isomorphic to $M_{n_i}^{B_i}$ as $W$-algebras, where $M_{n_i}^{B_i}$ is the algebra $M_{n_i}(F)$ of matrices of order $n_i\geq 1$ over $F$  with the $W$-action given by left and right multiplication of a subalgebra $B_i\subseteq M_{n_i}(F)$ contaning $1_{M_{n_i}(F)}$.

    Suppose first that $n_i>1$ for some $1\leq i\leq m$. By Proposition \ref{pro: block triangularization} the subalgebra $B_i\subseteq M_{n_i}(F)$ is isomorphic to a subalgebra of $UT(t_1,\ldots,t_p)$
for some positive integers $(t_1,\ldots,t_p)$ such that $\sum_{l=1}^p t_l=n_i$.

 If $t_l=1$ for all $1\leq l \leq p$, then the $W$-algebra $UT_{n_i}^{B_i}$ of $n_i\times n_i$ upper triangular matrices with the $W$-action given by left and right multiplication of $B_i$ is a $W$-subalgebra of $M_{n_i}^{B_i}$. Hence, $\V^W(UT_{n_i}^{B_i})\subseteq\V^W(M_{n_i}^{B_i})$ and by Proposition \ref{pro: UT_n e UT_2}, at least one among $UT_2^F$ and $UT_2^D$ belongs to $\V^W(M_{n_i}^{B_i})$. But since $M_{n_i}^{B_i}\in \V^W(\Bar{A})\subseteq \V^W(A)$ we reach a contradiction.

So, let us assume that $t_l\geq 2$ for some $1\leq l \leq p.$  We endow $B_i$  with a natural $W$-algebra structure by defining the action of $W$ on $B_i$ as the action of $B_i$ on itself via left and right multiplication. Thus $B_i\in \V^W(M_{n_i}^{B_i})$. Moreover, by Theorem \ref{thm: semisimple action and algebra with 1}, we may assume without loss of generality that only the semisimple part $S_{B_i}\cong M_{t_1}(F)\oplus \cdots \oplus M_{t_p}(F)$ of $B_i$ acts on $B_i$ by left and right multiplication.
Since $M_{t_l}(F)$ is a minimal ideal of $S_{B_i},$ it is stable under left and right multiplication.  Hence, the $W$-algebra $M_{t_l}$, i.e., the algebra $M_{t_l}(F)$ equipped with the $W$-action given by left and right multiplication of $M_{t_l}(F)$ on itself, is a $W$-subalgebra of $B_i$.
 Consequently, since $B_i\in\V^W(M_{n_i}^{B_i}) \subseteq \V^W(\Bar{A})\subseteq \V^W(A)$,  it follows that $M_{t_l}\in \V^W(A)$, a contradiction.

    Assume now that $n_i=1$ for all $1\leq i \leq m$, that is, $\Bar{A}_i$ is isomorphic to $F$ for all $i$. As a result, by Wedderburn-Malcev Theorem for associative algebras (see \cite[Theorem 3.4.3]{GZbook}), we have that
    $$
    A=A_1\oplus \cdots \oplus A_m + J,
    $$
    where $A_1\cong \cdots \cong A_m \cong F$ (as ordinary algebras). Thus, by Theorem \ref{thm: exponent} to complete the proof, it is enough to show that $A_r JA_s=0$  for all $1 \leq r,s\leq m$, $r\neq s$.  
    
    Suppose to the contrary that there exist $1\leq r,s\leq  m$, $r \neq s$, such that $A_r JA_s\neq 0$. Then, by Lemma  \ref{lem: AiJAk in J^q}  it follows that either $UT_2^F\in \V^W(A)$ or $UT_2^D\in \V^W(A)$, a contradiction and the lemma is proved.	
\end{proof}

Now, we are in a position to characterize the generalized varieties of polynomial growth.

\begin{theorem}
    	Let $W$ be a unital algebra over a field $F$ of characteristic zero and $A$ be a finite dimensional $W$-algebra. Then the sequence $c_n^W(A)$, $n\geq 1$, is polynomially bounded if and only if $UT_2^F, UT_2^D, M_k \notin \V^W(A)$ for all $k\geq 2$.
\end{theorem}
\begin{proof}
    It is clear that if $c_n^W(A)$ is polynomially bounded, then $UT_2^F, UT_2^D, M_k \notin \V^W(A)$, for $k\geq 2$, since, by Theorems \ref{identita e codimensioni di UT2F}, \ref{identita e codimensioni di UT2D} and \ref{identitadellematriciecodimensioni}, $UT_2^F$, $UT_2^D$ and $M_k$, for $k\geq 2$, generate generalized varieties of exponential growth.

    Now, assume that $UT_2^F, UT_2^D , M_k\notin \V^W(A)$ for all $k\geq 2$. Let  $\Phi$ denote the acting homomorphism of $W$ in $A$, and let $\overline{W}$ denote the subalgebra $\Phi(W)$ of $\M(A).$ We can then consider the semi-direct product $\overline{W} \ltimes A$ that by Proposition \ref{prop: semi-direct product} is a unital $W$-algebra. Since both $\overline{W}$ and $A$ are finite-dimensional, it follows that $\overline{W} \ltimes A$ is also finite-dimensional.
    Thus, by Lemma \ref{lem:unital W-algebra PG} we have that $c_n^W(\overline{W}\ltimes A)$ is polynomially bounded. Additionally,  since by Proposition \ref{prop: semi-direct product} we can identify $A$ with a suitable $W$-subalgebra  of $\overline{W}\ltimes A$, it follows that  $\V^W(A) \subseteq \V^W(\overline{W}\ltimes A).$ Therefore, $c^W_n(A)\leq c^W_n(\overline{W}\ltimes A)$ and consequently, $c_n^W(A)$ is also polynomially bounded, as required.
\end{proof}

As a consequence, we have the following characterization of generalized varieties of almost polynomial growth.

\begin{corollary}
$UT_2^F$, $UT_2^D$ and $M_k$, for $k\geq 2$, are the only finite dimensional $W$-algebras generating generalized varieties of almost polynomial growth.
\end{corollary}

\section{On the Specht property}

This final section deals with the Specht property of generalized $T$-ideals showing a conjecture and some future directions.

We say that a class of algebras has the Specht property if any $T$-ideal of that class is finitely generated. For example, a celebrated theorem of Kemer states that the class of associative algebras in characteristic zero has such property (see \cite{Kemer1985}). This result represents one of the milestones in the theory of algebras with polynomial identities, so that many mathematicians have continued to follow this line of research, attempting to solve the Specht problem in other contexts as well, such as algebras with additional structures, non-associative algebras and so on (see for instance \cite{AljadeffGiambrunoKarasik2017, AljadeffKanel2010, VaisZelmanov1989}). All the affirmative answers are given in characteristic zero, since in positive characteristic there are several examples, even for associative algebras, of infinitely generated $T$-ideals (see \cite{Belov1999, Grishin1999, Shchigolev1999}).

Throughout the paper, we have always considered finite dimensional $W$-algebras, although $W$ had no restrictions on the dimension or on the number of generators. Regardless, a keen observer may notice that Theorems \ref{identitadellematriciecodimensioni}, \ref{identita e codimensioni di UT2F}, \ref{identita e codimensioni di UT2D} and \ref{identita e codimensioni di UT2full} raise an interesting matter regarding the number of generators of a $T_W$-ideal. In general, is it or is it not finitely generated? 

Furthermore, working with infinite dimensional algebras can lead to other unusual situations. For instance, if $W$ has infinite dimension, then $P_n^W$ is also infinite dimensional and so can be $P_n^W(A).$ In other terms, this means that the generalized codimensions can be infinite, as already anticipated in Section \ref{sezione4}.

It is clear that such ``pathological" behavior of the codimensions may appear provided that both $W$ and $A$ are infinite dimensional, since if $A$ has finite dimension, then the action of $W$ must be finite,  i.e., the dimension of $\Phi(W)$ is finite, as in the case of $UT_2$ for example.

Concerning the Specht property, it seems that the crucial aspect is not the dimension of $W,$ but rather its being finitely generated. More precisely, we state the following conjecture.

\begin{conjecture}\label{congettura}
    Let $A$ be a $W$-algebra over a field of characteristic zero. If $W$ is finitely generated, then $\I^W(A)$ is finitely generated.
\end{conjecture}

To support the previous conjecture and simultaneously provide an example of infinite codimensions, we now compute the $T_W$-ideal of generalized identities of the unital Grassmann algebra $E$, assuming that either $\Phi(W)\cong E$ or $\Phi(W)\cong E_k,$ the finite dimensional Grassmann algebra.
To achieve this, we assume that $W$ is not finitely generated. 

Recall that the Grassmann algebra $E$ with unity is the algebra generated over $F$ by the elements $1,$ $e_1,e_2,\ldots$ such that $e_ie_j=-e_je_i.$ Moreover, a basis of $E$ consists of all the elements $g =e_{i_1}e_{i_2}\cdots e_{i_n}$ such that $i_1<i_2<\cdots<i_n$ and $n\geq 0.$ Here, we are assuming that if $n=0$ then $g=1.$ It is clear that by construction $E$ is not finitely generated.

One can also consider a finite set of generators and construct the finite dimensional Grassmann algebra, which is a subalgebra of $E.$ More precisely, $E_k$ is the algebra generated by $1,$ $e_1,e_2,\ldots , e_k$ and $\dim_F E_k= 2^k.$

Since $W$ is not finitely generated and there is no loss of generality in considering $E,$ we set $W=E$ for simplicity of notation.

We start computing the $T_E$-ideal of $E^{(k)}$, i.e., $E$ with the action of $E_k$ by left and right multiplication, for all $k\geq 1.$
A simple computation shows that the polynomials 
\begin{equation}\label{generatoridigrassmann}
    [x_1,x_2,x_3], \quad [e_i,x_1,x_2], \quad e_jx, \quad xe_j
\end{equation}
for all $1\leq i\leq k$ and for all $j\geq k+1,$ belong to $\I^E(E^{(k)}).$ Moreover, notice that the polynomial $[x,g]=0,$ where $g = e_{i_1}\cdots e_{i_{2r}}$ lies in the center of $E,$ and $[e_i,x,e_j]=0$ are $E$-trivial polynomials in the sense of \cite[Corollary 2.3]{MartinoRizzo2024}.
Also, by straightforward calculations we have the equality
\begin{equation}\label{unpolinomiotrivialedigrassmann}
    [e_i,xe_j] = [e_i,x]e_j+2xe_ie_j.
\end{equation}

\begin{lemma}\label{conseguenzeGrassmann}
    The following generalized polynomials are consequences of polynomials in \eqref{generatoridigrassmann}:
    \begin{equation*}
        \begin{split}
            &[x_1,x_2][x_3,x_4]+[x_1,x_4][x_3,x_2] \\
            &[e_i,x_2][x_3,x_4]+[e_i,x_4][x_3,x_2] \\
            & [x_1,x_2,e_i] \\
            & [e_i,x_1][e_j,x_2] + 2[x_1,x_2]e_ie_j
        \end{split}
    \end{equation*}
    for all $i,j\geq1.$
\end{lemma}
\begin{proof}
    Following the lines of \cite[Theorem 4.1.8]{GZbook}, one can prove that $[x_1,x_2,x_3]$ implies $[x_1,x_2][x_3,x_4]+[x_1,x_4][x_3,x_2]$ and $[e_i,x_1,x_2]$ implies $[e_i,x_2][x_3,x_4]+[e_i,x_4][x_3,x_2].$ Moreover, using Jacobi identity, it is clear that $[e_i,x_1,x_2]$ implies also $[x_1,x_2,e_i].$

    Now, take $[e_i,x_1,x_2]$ and substitute $x_1e_j$ instead of $x_1.$ Using $[x,g]=0,$ \eqref{unpolinomiotrivialedigrassmann} and $[e_i,x_1,x_2]\in \I^E(E^{(k)}),$ we get that 
    \begin{equation*}
        \begin{split}
             [e_i,x_1e_j, x_2]&= [e_i,x_1e_j]x_2-x_2[e_i,x_1e_j] = [e_i,x_1]e_jx_2+2x_1e_ie_jx_2-x_2[e_i,x_1]e_j-2x_2x_1e_ie_j \\
            &=[e_i,x_1]e_jx_2 -x_2[e_i,x_1]e_j+2x_1x_2e_ie_j-2x_2x_1e_ie_j 
            = [e_i,x_1]e_jx_2-x_2[e_i,x_1]e_j +2[x_1,x_2]e_ie_j \\
            &\equiv [e_i,x_1]e_jx_2 - [e_i,x_1]x_2 e_j +2[x_1,x_2]e_ie_j = [e_i,x_1][e_j,x_2] +2[x_1,x_2]e_ie_j.
        \end{split}
    \end{equation*}
    Thus $[e_i,x_1,x_2]$ imply $[e_i,x_1][e_j,x_2] + 2[x_1,x_2]e_ie_j$ and we are done.
\end{proof}

We are now in a position to compute a basis of $\I^E(E^{(k)}).$ In the next theorem, recall that the \emph{support} of $g\in E,$ denoted by $\supp(g),$ is the set of all the $e_i$'s appearing in $g.$ Moreover, we denote by $E_1$ the span over $F$ of all the words in the generators of $E$ of odd length.
 
\begin{theorem}\label{identitadiEk}
    Let $E^{(k)},$ with $k\geq 1$, be the unital Grassmann algebra $E$ over a field $F$ of characteristic zero regarded as an $E$-algebra, where $E$ acts as the subalgebra $E_k$ by left and right multiplication.
    Then $\I^E(E^{(k)})$ is generated, as $T_E$-ideal, by the polynomials 
    \begin{equation*}
    [x_1,x_2,x_3], \quad [e_i,x_1,x_2], \quad e_jx, \quad xe_j
\end{equation*}
for all $1\leq i\leq k$ and for all $j\geq k+1.$ Moreover, $c^E_n(E^{(k)}) = 2^{n-1}(2^{k+1}-1)$ for all $n\geq 1.$
\end{theorem}
\begin{proof}
    Let $I$ be the $T_E$-ideal generated by the above generalized polynomials. A straightforward computation shows that $I\subseteq \I^E(E^{(k)}).$ Thus in order to prove the opposite inclusion, let $f\in\I^E(E^{(k)})$ be a multilinear generalized polynomial of degree $n$ and suppose, as we may, that in $f$ does not appear any $e_j$ for all $j\geq k+1.$ 

    By the Poincaré-Birkhoff-Witt Theorem and by $[x_1,x_2,e_i]\equiv 0$, $f$ can be written as a linear combination of polynomials of the type
    $$
    x_{c_1}\cdots x_{c_t}v_1v_2\cdots v_me_{l_1}\cdots e_{l_s},
    $$
    where $c_1<\cdots < c_t,$ $l_1<\cdots <l_s,$ $0\leq s\leq k$ and $v_1,\ldots, v_m$ are left-normed commutators containing variables and $e_i$'s. Moreover, since $E$ has unity, without loss of generality in what follows we can disregard the tail $x_{c_1}\cdots x_{c_t}$ and suppose that the variables $x_1,\ldots,x_n$ in $f$ appear exclusively within the commutators.
    
    Now, using Lemma \ref{identitadiEk} and the $W$-trivial polynomials mentioned above, we can write $f$ as linear combination of the polynomials
    \begin{equation}\label{polinomigradopari}
    [x_{1},x_{2}][x_3,x_4]\cdots [x_{2m-1},x_{2m}]e_{l_1}\cdots e_{l_s},
    \end{equation}
    if $n=2m$ is even or 
    \begin{equation}\label{polinomigradodispari}
    [e_{l_1},x_{1}][x_2,x_3]\cdots [x_{2m},x_{2m+1}]e_{l_2}\cdots e_{l_s},
    \end{equation}
    if $n=2m+1$ is odd. In both cases $l_1<\cdots <l_s$ and $0\leq s\leq k.$ This implies that such polynomials generate $P^W_n$ modulo the generalized identities of $E.$

    We claim that they are also linearly independent modulo $\I^E(E^{(k)}).$ To this end, first suppose that $\deg f= 2m$ and write
    $$
    f=  \sum_{L}\alpha_L[x_{1},x_{2}][x_3,x_4]\cdots [x_{2m-1},x_{2m}]e_{l_1}\cdots e_{l_s} \quad(\text{mod }\I^E(E^{(k)})),
    $$
    where $L=\{l_1,\ldots, l_s\}.$ Let $s$ be the smallest integer for which by contradiction there exists $L$ such that $\alpha_L\neq 0.$ Then, by making the evaluation $x_1\mapsto g_1= e_{j_1}\cdots e_{j_t}g'$ and $x_i\mapsto g_i$ for all $2\leq i\leq 2m,$ where $\{j_1,\ldots, j_t\}\uplus L=\{1,\ldots, k\},$ $g'\in E$ such that $g_1\in E_1,$ $\supp(g')\cap \{j_1,\ldots, j_t\}=\emptyset,$ $g_2,\ldots, g_{2m}\in E_1$ and $\supp(g_i)\cap\supp(g_j)=\emptyset$ for all $i\neq j,$ we get $2^m \alpha_L g_1g_2\cdots g_{2m}e_{l_1}\cdots e_{l_s}=0 $ that is $\alpha_L=0,$ a contradiction. Therefore, $f=0$ and we are done in this case.

    Now suppose that $\deg f=2m+1$ and write
    $$
    f=  \sum_{L}\alpha_L[e_{l_1},x_{1}][x_2,x_3]\cdots [x_{2m},x_{2m+1}]e_{l_2}\cdots e_{l_s} \quad(\text{mod }\I^E(E^{(k)})),
    $$
    where $L=\{l_1,\ldots, l_s\}.$ With the same evaluation as before, one can get $f=0$ also in this case. 
    
    Thus these polynomials are linearly independent modulo $\I^E(E^{(k)})$ and since $P^E_n\cap I\subseteq P_n^E\cap\I^E(E^{(k)}),$ this proves that $I=\I^E(E^{(k)})$ and the polynomials in \eqref{polinomigradopari} and \eqref{polinomigradodispari} form a basis of $P_n^E$ modulo $P_n^E\cap\I^E(E^{(k)}).$

    If we set $\gamma_{2m}(E)$ and $\gamma_{2m+1}(E)$ to be the number of polynomials in \eqref{polinomigradopari} and \eqref{polinomigradodispari}, respectively, then it is clear that $\gamma_{2m}(E)=2^k,$ $\gamma_{2m+1}(E)= 2^k-1$ for all $m\geq 1$ and
    $$
    c_n^E(E^{(k)})=\sum_{t=0}^n\binom{n}{t}\gamma_t(E).
    $$
    Therefore by counting
    \begin{equation*}
    \begin{split}
        c_n^E(E^{(k)}) &= \sum_{h=0}^{\lfloor \frac{n}{2}\rfloor}\binom{n}{2h}2^k+\sum_{h=0}^{\lfloor \frac{n-1}{2}\rfloor}\binom{n}{2h+1}(2^k-1) = 2^k\left[\sum_{h=0}^{\lfloor \frac{n}{2}\rfloor}\binom{n}{2h}+\sum_{h=0}^{\lfloor \frac{n-1}{2}\rfloor}\binom{n}{2h+1}\right]- \sum_{h=0}^{\lfloor \frac{n-1}{2}\rfloor}\binom{n}{2h+1}
        \\
        &= 2^{n-1}(2^{k+1}-1),
    \end{split}
    \end{equation*}
    as claimed.
\end{proof}

As a consequence, we have the following.

\begin{corollary}\label{cor: exp of E}
    The limit $ \exp^E(E^{(k)})=\lim_{n\to \infty}\sqrt[n]{c_n^E(E^{(k)})}$ exists, and precisely $\exp^E(E^{(k)})=2$.
\end{corollary}

With similar techniques we can compute also the $T_W$-ideal of generalized identities of $E$ when $W$ acts by left and right multiplication as the full algebra $E.$

\begin{theorem}\label{identitadiE}
    Let $E$ be the unital Grassmann algebra over a field $F$ of characteristic zero regarded as an $E$-algebra, where $E$ acts on itself by left and right multiplication. Then $\I^{E}(E)$ is generated, as $T_E$-ideal, by the polynomials 
    \begin{equation}\label{generatoriE}
    [x_1,x_2,x_3], \quad [e_i,x_1,x_2],
\end{equation}
for all $i\geq 1.$ Moreover, $c^{E}_n(E) = +\infty$ for all $n\geq 1.$
\end{theorem}
\begin{proof}
    Following the lines of the previous Theorem, one can prove that modulo $\I^E(E),$ any multilinear polynomial $f$ can be written as linear combination of
    \begin{equation*}
    [x_{1},x_{2}][x_3,x_4]\cdots [x_{2m-1},x_{2m}]e_{l_1}\cdots e_{l_s},
    \end{equation*}
    if $n=2m$ is even or 
    \begin{equation*}
    [e_{l_1},x_{1}][x_2,x_3]\cdots [x_{2m},x_{2m+1}]e_{l_2}\cdots e_{l_s},
    \end{equation*}
    if $n=2m+1$ is odd. In both cases $l_1<\cdots <l_s$, but here we have no limit on the value of $s.$ It is also easily proven that these polynomials are linearly independent modulo $\I^E(E),$ thus the polynomials in \eqref{generatoriE} generate $\I^E(E)$ and consequently $c_n^E(E)=+\infty$ for all $n\geq 1.$
\end{proof}

\begin{corollary}
    The $T_E$-ideal of generalized polynomial identities of the Grassmann algebra $E$ with either the action by left and right multiplication of $E_k,$ for any $k,$ or the action of $E,$ has not the Specht property.
\end{corollary}
\begin{proof}
    The claim follows directly from Theorems \ref{identitadiEk} and \ref{identitadiE}. In fact, concerning $\I^E(E^{(k)}),$ it is clear that for instance the identity $e_jx$ for any fixed $j\geq k+1$ can not follow from the remaining ones. Similarly, for $\I^E(E),$ any identity $[e_i,x_1,x_2]$ can not follow from $[x_1,x_2,x_3]$ since the latter one has degree 3 and, by definition, $E$ is not embedded in $E\langle X\rangle$. Also, by the multiplication rules in $E,$ it can not be followed by any other identity of the type $[e_j,x_1,x_2]$ with $j\neq i.$
\end{proof}

Notice that the previous corollary shows that the Specht property may fail whenever $W$ is not finitely generated, regardless of its action on $A,$ strengthening Conjecture \ref{congettura}.

\smallskip

Finally, note that Theorems \ref{identitadiEk} and \ref{identitadiE} suggest that the decisive factor determining the behavior of the codimensions is neither the dimension of $W$ itself nor that of $A$, but rather the finiteness of the action of $W$ on $A$, that is, the finiteness of the dimension of $\Phi(W)$. Moreover, as a consequence of Corollary \ref{cor: exp of E}, we find that the $W$-exponent of $E^{(k)}$ is equal to the ordinary exponent of $E$. Therefore, we formulate the following conjecture.

\begin{conjecture}\label{congettura2}
    Let $W$ be a unital $F$-algebra over a field of characteristic zero. If $A$ is a $W$-algebra such that the action of $W$ on $A$ is finite, then the limit 
    $$\exp^W(A)=\lim_{n\to \infty}\sqrt[n]{c_n^W(A)}$$
    exists and is a nonnegative integer. Furthermore, 
   $$\exp^W(A)=\exp(A).$$
\end{conjecture}

\

\

\noindent\textbf{Acknowledgements:} We are grateful to Professor Michael Dokuchaev for suggesting the study of multiplier algebras.

\

\

\noindent \textbf{Competing interests} The authors declare no competing interests.

\end{document}